\documentclass[10pt,a4paper]{article}

\usepackage{latexsym,amsfonts,amsmath,amssymb,amsthm,mathrsfs,color}

\newtheorem{theorem}{Theorem}
\newtheorem{lemma}{Lemma}

\newtheorem{remark}{Remark}
\newtheorem{proposition}{Proposition}
\newtheorem{definition}{Definition}
\newtheorem{corollary}{Corollary}

\newcommand{\rd}{\mathrm{d}}
\newcommand{\bszero}{\boldsymbol{0}}
\newcommand{\bsc}{\boldsymbol{c}}
\newcommand{\bsk}{\boldsymbol{k}}
\newcommand{\bsl}{\boldsymbol{l}}
\newcommand{\bsr}{\boldsymbol{r}}
\newcommand{\bsx}{\boldsymbol{x}}
\newcommand{\bsy}{\boldsymbol{y}}
\newcommand{\bsw}{\boldsymbol{w}}
\newcommand{\bsalpha}{\boldsymbol{\alpha}}
\newcommand{\bsgamma}{\boldsymbol{\gamma}}
\newcommand{\bssigma}{\boldsymbol{\sigma}}
\newcommand{\FF}{\mathbb{F}}
\newcommand{\NN}{\mathbb{N}}
\newcommand{\RR}{\mathbb{R}}

\newcommand{\Hcal}{\mathcal{H}}
\newcommand{\Kcal}{\mathcal{K}}
\newcommand{\rms}{\mathrm{rms}}

\newcommand{\wor}{\mathrm{wor}}

\allowdisplaybreaks

\begin{document}

\title{Optimal order quasi-Monte Carlo integration in weighted Sobolev spaces of arbitrary smoothness
\thanks{The work of T.~G. is supported by JSPS Grant-in-Aid for Young Scientists No.15K20964.
The work of K.~S. and T.~Y. is supported by the Program for Leading Graduate Schools, MEXT, Japan and Australian Research Council's Discovery Projects funding scheme (project number DP150101770).
The work of K.~S. is also supported by Grant-in-Aid for JSPS Fellows No.15J05380.
}}

\author{Takashi Goda\thanks{Graduate School of Engineering, The University of Tokyo, 7-3-1 Hongo, Bunkyo-ku, Tokyo 113-8656, Japan (\tt{goda@frcer.t.u-tokyo.ac.jp})}, Kosuke Suzuki\thanks{School of Mathematics and Statistics, The University of New South Wales, Sydney 2052, Australia ({\tt kosuke.suzuki1@unsw.edu.au})}, Takehito Yoshiki\thanks{School of Mathematics and Statistics, The University of New South Wales, Sydney 2052, Australia ({\tt takehito.yoshiki1@unsw.edu.au})}}

\date{\today}

\maketitle

\begin{abstract}
We investigate quasi-Monte Carlo integration using higher order digital nets in weighted Sobolev spaces of arbitrary fixed smoothness $\alpha \in \NN$, $\alpha \ge 2$, defined over the $s$-dimensional unit cube. We prove that randomly digitally shifted order $\beta$ digital nets can achieve the convergence of the root mean square worst-case error of order $N^{-\alpha}(\log N)^{(s-1)/2}$ when $\beta \ge 2\alpha$. The exponent of the logarithmic term, i.e., $(s-1)/2$, is improved compared to the known result by Baldeaux and Dick, in which the exponent is $s\alpha /2$. Our result implies the existence of a digitally shifted order $\beta$ digital net achieving the convergence of the worst-case error of order $N^{-\alpha}(\log N)^{(s-1)/2}$, which matches a lower bound on the convergence rate of the worst-case error for any cubature rule using $N$ function evaluations and thus is best possible.
\end{abstract}

\section{Introduction and the main result}
In this paper we investigate quasi-Monte Carlo (QMC) integration of functions defined over the $s$-dimensional unit cube. For an integrable function $f\colon [0,1)^s\to \RR$, we denote the true integral of $f$ by
\begin{align*}
I(f) := \int_{[0,1)^s}f(\bsx)\, \rd \bsx .
\end{align*}
QMC integration of $f$ over a finite point set $P\subset [0,1)^s$ approximates $I(f)$ by
\begin{align*}
I(f;P) := \frac{1}{|P|}\sum_{\bsx\in P}f(\bsx) .
\end{align*}
Here $P$ is a multiset and so if an element occurs multiple times it is counted according to its multiplicity. The key ingredient for success of QMC integration is to construct good point sets depending on a function class to which $f$ belongs. In the classical QMC theory, for instance, a class of functions with bounded variation in the sense of Hardy and Krause has been often considered \cite{KNbook,Nbook}. For this function class, the Koksma-Hlawka inequality states that the integration error is bounded by
\begin{align*}
|I(f;P)-I(f)| \le V_{\mathrm{HK}}(f)D^*(P) ,
\end{align*}
where $V_{\mathrm{HK}}(f)$ denotes the variation of $f$ in the sense of Hardy and Krause, and $D^*(P)$ the star-discrepancy of $P$, see \cite[Chapter~3]{Nbook}. This inequality motivates construction of low-discrepancy point sets. We refer to \cite[Chapter~8]{DPbook} for several explicit constructions of point sets whose star-discrepancy is of order $N^{-1}(\log N)^{s-1}$, where $N$ denotes the number of points, i.e., $N=|P|$.

One of the recent interest in the research community is to consider a function class consisting of smooth functions and to construct good point sets which achieve higher order convergence of the integration error in that function class, see for instance \cite{Dick08,HMOT15,Mark13,Teml03}. A function space of our particular interest in this paper is a weighted Sobolev space $\Hcal_{\alpha,\bsgamma}$ of arbitrary fixed smoothness $\alpha$ for a set of non-negative real numbers $\bsgamma=(\gamma_u)_{u\subseteq \{1,\ldots,s\}}$. Here $\alpha\geq 2$ is a positive integer. (We shall give the precise definition of $\Hcal_{\alpha,\bsgamma}$ in Subsection~\ref{subsec:sobolev}.) This space has been studied, for instance, in \cite{BD09,Dick09,DPbook} in the context of QMC integration. The breakthrough in this research direction was made by Dick and his collaborators \cite{BD09,BDP11,Dick07,Dick08,Dick09,DPbook}, who provide us with an explicit construction of good point sets called \emph{higher order digital nets} achieving almost optimal convergence of the integration error of order $N^{-\alpha}(\log N)^{c(s,\alpha)}$ for some $c(s,\alpha)>0$. (We shall give the definition of higher order digital nets in Subsection~\ref{subsec:ho_digital_net}.) The above order of convergence $\alpha$ is best possible up to some power of a $\log N$ factor.

A thorough analysis on the exponent $c(s,\alpha)$ has been recently done for periodic Sobolev spaces and periodic Nikol'skij-Besov spaces with dominating mixed smoothness in \cite{HMOT15}. They obtained $c(s,\alpha)=(s-1)/2$ for order 2 digital nets in the former spaces for instance. Although the result is best possible, there are restrictions that only periodic function spaces are taken into account and that the smoothness parameter $\alpha$, which equals $r$ in their notation and is considered to be a positive real number, should be less than 2. Thus, the question arises whether higher order digital nets can achieve the best possible convergence of the integration error in non-periodic function spaces of $\alpha \ge 2$. In this paper we give an affirmative answer to this question.

To state the main result of this paper, we introduce some notation here. Let $\NN$ be the set of positive integers and $\NN_0:=\NN\cup \{0\}$. Let $b$ be a prime and $\FF_b$ the finite field with $b$ elements, which is identified with the set $\{0,1,\ldots,b-1\}$ equipped with addition and multiplication modulo $b$. For $x\in [0,1)$, its $b$-adic expansion $x=\sum_{i=1}^{\infty}\xi_ib^{-i}$ with $\xi_i\in \FF_b$ is understood to be unique in the sense that infinitely many of the $\xi_i$'s are different from $b-1$. The operator $\oplus$ denotes digitwise addition modulo $b$, that is, for $x=\sum_{i=1}^{\infty}\xi_ib^{-i}\in [0,1),x'=\sum_{i=1}^{\infty}\xi'_ib^{-i}\in [0,1)$, we define 
\begin{align*}
x\oplus x':= \sum_{i=1}^{\infty}\frac{\eta_i}{b^i}\quad \text{with}\quad \eta_i=\xi_i+\xi'_i \pmod b.
\end{align*}
Note that $x\oplus x'$ is not always defined via its unique $b$-adic expansion and even may equal $1\notin [0,1)$.
Such an instance is given by setting $b=2$, $x=2^{-1}+2^{-3}+2^{-5}+\cdots$ and $x'=2^{-2}+2^{-4}+2^{-6}+\cdots$.
However, if either $x$ or $x'$ can be written in a finite $b$-adic expansion, this situation never occurs.

Moreover, let $V$ be a normed function space with norm $\lVert \cdot\rVert_V$. The worst-case error of QMC integration over $P$ in $V$ is defined as
\begin{align*}
e^{\wor}(V;P) := \sup_{\substack{f\in V\\ \lVert f\rVert_V\le 1}}|I(f;P)-I(f)|.
\end{align*}
For $\bssigma\in [0,1)^s$, we write $P\oplus \bssigma:=\{\bsx\oplus \bssigma: \bsx\in P\}$, where $\oplus$ is applied componentwise.
Since we shall only consider a point set $P$ whose each element $\bsx$ can be written in finite $b$-adic expansions in this paper, $\bsx\oplus \bssigma$ is always defined via unique $b$-adic expansions.
Then the root mean square (RMS) worst-case error of QMC integration over $P\oplus \bssigma$ in $V$ with respect to a randomly chosen $\bssigma\in [0,1)^s$ is defined as
\begin{align*}
e^{\rms\text{--}\wor}(V;P) := \left( \int_{[0,1)^s}\left( e^{\wor}(V;P\oplus \bssigma) \right)^2\, \rd \bssigma\right)^{1/2}.
\end{align*}

Now the main result of this paper is given as follows.
\begin{theorem}\label{thm:main_result}
For $\alpha,\beta,m\in \NN$ and $t\in \NN_0$ with $\alpha\ge 2$, $\beta \ge 2\alpha$ and $0\le t\le \beta m$, let $P$ be an order $\beta$ digital $(t,m,s)$-net over $\FF_b$.
Let $\bsgamma=(\gamma_u)_{u\subseteq \{1,\ldots,s\}}$ be a set of the weights.
The RMS worst-case error of QMC integration over $P\oplus \bssigma$ in $\Hcal_{\alpha,\bsgamma}$ with respect to a randomly chosen $\bssigma\in [0,1)^s$ is bounded by
\begin{align}\label{eq:main_result}
 e^{\rms\text{--}\wor}(\Hcal_{\alpha,\bsgamma};P) \le \frac{1}{N^{\alpha}}\sum_{\emptyset \ne u\subseteq \{1,\ldots,s\}}\gamma_u^{1/2} C_{\alpha,\beta,b,t,u}(\log N)^{(|u|-1)/2},
\end{align}
where $C_{\alpha,\beta,b,t,u}>0$ for all $\emptyset \ne u\subseteq \{1,\ldots,s\}$ and $N=|P|=b^m$.
\end{theorem}
\noindent Note that the explicit form of $C_{\alpha,\beta,b,t,u}$ can be found later in (\ref{eq:constant_full}). This result directly implies the following.
\begin{corollary}\label{cor:existence}
For $\alpha,\beta,m\in \NN$ and $t\in \NN_0$ with $\alpha\ge 2$, $\beta \ge 2\alpha$ and $0\le t\le \beta m$, let $P$ be an order $\beta$ digital $(t,m,s)$-net over $\FF_b$.
Let $\bsgamma=(\gamma_u)_{u\subseteq \{1,\ldots,s\}}$ be a set of the weights.
There exists a $\bssigma\in [0,1)^s$ such that the worst-case error of QMC integration over $P\oplus \bssigma$ in $\Hcal_{\alpha,\bsgamma}$ is bounded by
\begin{align*}
 e^{\wor}(\Hcal_{\alpha,\bsgamma};P\oplus \bssigma) \le \frac{1}{N^{\alpha}}\sum_{\emptyset \ne u\subseteq \{1,\ldots,s\}}\gamma_u^{1/2} C_{\alpha,\beta,b,t,u}(\log N)^{(|u|-1)/2}.
\end{align*}
\end{corollary}

Although the $t$-value and thus the constants $C_{\alpha,\beta,b,t,u}$ may depend on $m$, it was shown by Dick \cite{Dick07,Dick08} that for large $m$ we can explicitly construct an order $\beta$ digital $(t,m,s)$-net with its $t$-value independent of $m$, see also Remark~\ref{rem:ho_digital_net}. Therefore, the rate of convergence which we obtain in this paper is of order $N^{-\alpha}(\log N)^{(s-1)/2}$ unless $\gamma_{\{1,\ldots,s\}}=0$. This compares favorably with what was obtained by Baldeaux and Dick in \cite[Theorem~24]{BD09}, where they considered the case where $P$ is an order $\alpha$ digital net over $\FF_b$, i.e., the case where the order of digital nets and the smoothness parameter coincide, and obtained a similar bound on the RMS worst-case error but with the exponent of the logarithmic term equal to  $s\alpha/2$. Our result shows that the exponent of the logarithmic term is actually independent of $\alpha$. Here we note that the convergence of order $N^{-\alpha}(\log N)^{(s-1)/2}$ in a similar function space has been proven by using the Frolov cubature rule in conjunction with periodization strategy, see for instance \cite{Ullrich14}, which is not a QMC integration rule though.

Moreover, from the results of \cite{DNP14}, we can see that the above results in Theorem~\ref{thm:main_result} and Corollary~\ref{cor:existence} are best possible. Now let $P=\{\bsx_0,\ldots,\bsx_{N-1}\}\subset [0,1)^s$ be an $N$ element point set and $\bsw=\{w_0,\ldots,w_{N-1}\}$ be an arbitrary real tuple. The worst-case error of cubature rule with points in $P$ and weights $\bsw$ is defined as
\begin{align*}
e^{\wor}(V;P,\bsw) := \sup_{\substack{f\in V\\ \lVert f\rVert_V\le 1}}\left| \sum_{n=0}^{N-1}w_nf(\bsx_n)-I(f)\right|.
\end{align*}
A lower bound on $e^{\wor}(V;P,\bsw)$ in the so-called half-period cosine space of smoothness $\alpha$ for the case of product weights, i.e., weights of the form $\gamma_u=\prod_{j\in u}\gamma_j$ for $\gamma_1,\ldots,\gamma_s>0$, was proven in \cite[Theorem~4]{DNP14}, and furthermore, it was shown in \cite[Theorem~1]{DNP14} that the half-period cosine space is continuously embedded in the Sobolev space $\Hcal_{\alpha,\bsgamma}$ which we consider in this paper. Combining these two results, we immediately have the following.
\begin{proposition}
Let $\alpha \ge 2$ be a positive integer and $\gamma_1,\ldots,\gamma_s>0$. For $u\subseteq \{1,\ldots,s\}$, let $\gamma_u=\prod_{j\in u}\gamma_j$ where the empty product equals 1. For any $N$ element point set $P\subset [0,1)^s$ and any real tuple $\bsw$ we have
\begin{align*}
 e^{\wor}(\Hcal_{\alpha,\bsgamma};P,\bsw) \ge c_{\alpha,\bsgamma,s}\frac{(\log N)^{(s-1)/2}}{N^{\alpha}},
\end{align*}
where $c_{\alpha,\bsgamma,s}$ is positive and independent of $P$ and $\bsw$.
\end{proposition}

This implies that the exponent $c(s,\alpha)$ cannot be less than $(s-1)/2$ in $\Hcal_{\alpha,\bsgamma}$ for the case of product weights with $\gamma_1,\ldots,\gamma_s>0$. Since Corollary~\ref{cor:existence} shows the existence of point sets which achieve exactly this order, our result is best possible. However, as our result is again an existence result and thus is not fully constructive, it is interesting to study an explicit construction of deterministic point sets which achieve the best possible convergence of the worst-case error in $\Hcal_{\alpha,\bsgamma}$. We leave it open for future work to address.

In the next section, we shall introduce the necessary background and notation such as weighted Sobolev spaces of smoothness $\alpha$ and higher order digital nets. In Section~\ref{sec:upper}, we shall prove Theorem~\ref{thm:main_result}, i.e., an upper bound on the RMS worst-case error of randomly digitally shifted order $\beta$ digital nets in $\Hcal_{\alpha,\bsgamma}$.
\section{Preliminaries}\label{sec:pre}

\subsection{Weighted Sobolev spaces}\label{subsec:sobolev}
First let us consider the one-dimensional unweighted case. The Sobolev space which we consider is given by
\begin{align*}
 \Hcal_{\alpha} & := \Big\{f \colon [0,1)\to \RR \mid f^{(r)} \colon \\
 & \qquad \quad \text{absolutely continuous for $r=0,\ldots,\alpha-1$}, f^{(\alpha)}\in L^2[0,1)\Big\},
\end{align*}
where $f^{(r)}$ denotes the $r$-th derivative of $f$. As in \cite[Section~10.2]{Wbook} this space is indeed a reproducing kernel Hilbert space with the reproducing kernel $\Kcal_{\alpha}\colon [0,1)\times [0,1)\to \RR$ and the inner product $\langle \cdot, \cdot \rangle_{\alpha}$ given as follows: 
\begin{align*}
 \Kcal_{\alpha}(x,y) = \sum_{r=0}^{\alpha}\frac{B_r(x)B_r(y)}{(r!)^2}+(-1)^{\alpha+1}\frac{B_{2\alpha}(|x-y|)}{(2\alpha)!} ,
\end{align*}
for $x,y\in [0,1)$, where $B_r$ denotes the Bernoulli polynomial of degree $r$, and
\begin{align*}
 \langle f, g \rangle_{\alpha} = \sum_{r=0}^{\alpha-1}\int_{0}^{1}f^{(r)}(x)\, \rd x \int_{0}^{1}g^{(r)}(x)\, \rd x + \int_{0}^{1}f^{(\alpha)}(x)g^{(\alpha)}(x)\, \rd x,
\end{align*}
for $f,g\in \Hcal_{\alpha}$.

Let us move on to the $s$-dimensional weighted case. In the following we write $\{1:n\}:=\{1,\ldots,n\}$ for $n\in \NN$.
Let $\bsgamma=(\gamma_u)_{u\subseteq \{1:s\}}$ be a set of non-negative real numbers which are called weights.
Note that the weights moderate the importance of different variables or groups of variables in function spaces and play an important role in the study of tractability \cite{SW98}.
However, such an investigation is out of the scope of this paper since we are interested in showing the optimal exponent of $\log N$ term in the error bound.
We consider the weighted function space for the sake of completeness.
Moreover, we shall use the following notation: For $v\subseteq \{1:s\}$ and $\bsx\in [0,1)^s$, let $\bsx_v=(x_j)_{j\in v}$. For $v\subseteq u\subseteq \{1:s\}$ and $\bsr_{u\setminus v}=(r_j)_{j\in u\setminus v}$, $(\bsr_{u\setminus v},\bsalpha_v,\bszero)$ denotes the $s$-dimensional vector whose $j$-th component is $r_j$ if $j\in u\setminus v$, $\alpha$ if $j\in v$, and $0$ otherwise. Now the weighted Sobolev space $\Hcal_{\alpha,\bsgamma}$ which we consider is the reproducing kernel Hilbert space whose reproducing kernel $\Kcal_{\alpha,\bsgamma}\colon [0,1)^s\times [0,1)^s\to \RR$ and inner product $\langle \cdot, \cdot \rangle_{\alpha,\bsgamma}$ are given as follows \cite{BD09}:
\begin{align*}
 \Kcal_{\alpha,\bsgamma}(\bsx,\bsy) = \sum_{u\subseteq \{1:s\}}\gamma_u \prod_{j\in u}\left\{\sum_{r=1}^{\alpha} \frac{B_r(x_j)B_r(y_j)}{(r!)^2}+(-1)^{\alpha+1}\frac{B_{2\alpha}(|x_j-y_j|)}{(2\alpha)!}\right\} ,
\end{align*}
for $\bsx=(x_1,\ldots,x_s),\bsy=(y_1,\ldots,y_s)\in [0,1)^s$, where the empty product always equals $1$, and
\begin{align*}
 \langle f, g \rangle_{\alpha,\bsgamma} & = \sum_{u\subseteq \{1:s\}}\gamma_u^{-1}\sum_{v\subseteq u}\sum_{\bsr_{u\setminus v}\in \{1:\alpha-1\}^{|u\setminus v|}} \\
 & \qquad \times \int_{[0,1)^{|v|}}\left(\int_{[0,1)^{s-|v|}}f^{(\bsr_{u\setminus v},\bsalpha_v,\bszero)}(\bsx)\, \rd \bsx_{\{1:s\}\setminus v}\right) \\
 & \qquad \quad \times \left(\int_{[0,1)^{s-|v|}} g^{(\bsr_{u\setminus v},\bsalpha_v,\bszero)}(\bsx) \, \rd \bsx_{\{1:s\}\setminus v}\right) \, \rd \bsx_v ,
\end{align*}
for $f,g\in \Hcal_{\alpha,\bsgamma}$, where for $u\subseteq \{1:s\}$ such that $\gamma_u=0$ we assume
\begin{align*}
 \sum_{v\subseteq u}\sum_{\bsr_{u\setminus v}\in \{1:\alpha-1\}^{|u\setminus v|}}& \int_{[0,1)^{|v|}} \left(\int_{[0,1)^{s-|v|}}f^{(\bsr_{u\setminus v},\bsalpha_v,\bszero)}(\bsx)\, \rd \bsx_{\{1:s\}\setminus v}\right) \\
 & \quad \times \left(\int_{[0,1)^{s-|v|}} g^{(\bsr_{u\setminus v},\bsalpha_v,\bszero)}(\bsx) \, \rd \bsx_{\{1:s\}\setminus v}\right) \, \rd \bsx_v = 0.
\end{align*}
Note that an integral and sum over the empty set is the identity operator and we formally set $0/0:=0$. 

\subsection{Higher order digital nets}\label{subsec:ho_digital_net}
Here we start with the general digital construction scheme of point sets as introduced by Niederreiter \cite{Nbook}.
\begin{definition}
For $m,n,s\in \NN$, let $C_1,\ldots,C_s\in \FF_b^{n\times m}$. Let $0\le h<b^m$ be an integer with its $b$-adic expansion $h=\sum_{i=0}^{m-1}\eta_i b^i$. For $1\le j\le s$, let us consider
\begin{align*}
 x_{h,j} = \frac{\xi_{1,h,j}}{b}+\frac{\xi_{2,h,j}}{b^2}+\cdots + \frac{\xi_{n,h,j}}{b^n} ,
\end{align*}
where $\xi_{1,h,j},\xi_{2,h,j},\ldots,\xi_{n,h,j}$ are given by
\begin{align*}
 (\xi_{1,h,j},\xi_{2,h,j},\ldots,\xi_{n,h,j})^{\top} = C_j (\eta_0,\eta_1,\ldots,\eta_{m-1})^{\top}.
\end{align*}
The set $P=\{\bsx_0,\bsx_1,\ldots,\bsx_{b^m-1}\}\subset [0,1)^s$ with $\bsx_h=(x_{h,1},\ldots,x_{h,s})$ is called a digital net over $\FF_b$ (with generating matrices $C_1,\ldots,C_s$).
\end{definition}
\noindent The dual net of $P$, denoted by $P^{\perp}$, is defined as follows.
\begin{definition}
For $m,n,s\in \NN$ and $C_1,\ldots,C_s\in \FF_b^{n\times m}$, let $P$ be a digital net over $\FF_b$ with generating matrices $C_1,\ldots,C_s$. The dual net of $P$ is defined as
\begin{align*}
 P^{\perp} := \{\bsk=(k_1,\ldots,k_s)\in \NN_0^s\colon C_1^{\top} \vec{k}_1\oplus \cdots \oplus C_s^{\top} \vec{k}_s = \bszero \in \FF_b^m\},
\end{align*}
where we set $\vec{k}:=(\kappa_0,\ldots,\kappa_{n-1})$ for $k\in \NN_0$ with its $b$-adic expansion $k=\kappa_0 +\kappa_1 b+\cdots $, which is actually a finite expansion.
\end{definition}

For $\alpha\in \NN$, we define a metric function $\mu_{\alpha}$ as follows.
\begin{definition}
Let $\alpha\in \NN$. For $k\in \NN$ with its $b$-adic expansion $k=\kappa_1b^{c_1-1}+\kappa_2b^{c_2-1}+\cdots+\kappa_vb^{c_v-1}$ such that $\kappa_1,\ldots,\kappa_v\in \{1,\ldots,b-1\}$ and $c_1>c_2>\cdots >c_v>0$. Then we define 
\begin{align*}
 \mu_{\alpha}(k):=\sum_{i=1}^{\min(\alpha,v)}c_i ,
\end{align*}
and $\mu_{\alpha}(0):=0$. For $\bsk=(k_1,\ldots,k_s)\in \NN_0^s$, we define
\begin{align*}
 \mu_{\alpha}(\bsk):=\sum_{j=1}^{s}\mu_{\alpha}(k_j).
\end{align*}
\end{definition}
\noindent Note that the above definition was originally given in \cite{Nied86,RT97} for the case $\alpha=1$ and in \cite{Dick07,Dick08} for $\alpha\ge 2$. We simply call $\mu_{\alpha}$ the \emph{Dick metric function} for any $\alpha \ge 1$ throughout this paper. Now we define the minimum Dick metric of a digital net, which shall play a critical role in the subsequent analysis.
\begin{definition}\label{def:minimum_Dick_metric}
Let $P$ be a digital net over $\FF_b$ and $P^{\perp}$ its dual net. For $\alpha\in \NN$, the minimum Dick metric of $P$ is defined as
\begin{align*}
 \delta_{\alpha}(P) := \min_{\bsk\in P^{\perp}\setminus \{\bszero\}}\mu_{\alpha}(\bsk). 
\end{align*}
\end{definition}

Now we give the definition of higher order digital nets.
\begin{definition}\label{def:ho_digital_net}
For $m,n,\alpha,s\in \NN$ with $n\ge \alpha m$, let $P$ be a digital net over $\FF_b$ with generating matrices $C_1,\ldots,C_s\in \FF_b^{n\times m}$. For $1\le i\le n$ and $1\le j\le s$, we denote by $\bsc_{i,j}\in \FF_b^m$ the $i$-th row vector of $C_j$. Let $t$ be an integer with $0\le t\le \alpha m$ which satisfies the following condition: For all $1\le i_{j,v_j}<\ldots < i_{j,1}\le n$ with 
\begin{align*}
 \sum_{j=1}^{s}\sum_{l=1}^{\min(\alpha,v_j)}i_{j,l}\le \alpha m -t,
\end{align*}
the vectors $\bsc_{i_{1,v_1},1},\ldots,\bsc_{i_{1,1},1}, \ldots, \bsc_{i_{s,v_s},s},\ldots,\bsc_{i_{s,1},s}$ are linearly independent over $\FF_b$. Then we call $P$ an order $\alpha$ digital $(t,m,s)$-net over $\FF_b$.
\end{definition}
\noindent The following property of order $\alpha$ digital $(t,m,s)$-nets directly follows from the linear independence of the rows of the generating matrices, that is, for any order $\alpha$ digital $(t,m,s)$-net $P$ over $\FF_b$, we have
\begin{align*}
 \delta_{\alpha}(P) > \alpha m -t. 
\end{align*}
Moreover, the following lemma is an obvious adaptation of the result shown in \cite[Theorem~3.3]{Dick07} and \cite[Theorem~4.10]{Dick08}, which states that any order $\alpha$ digital net is also an order $\alpha'$ digital net as long as $1\le \alpha' <\alpha$.
\begin{lemma}\label{lem:propagation}
For $\alpha\in \NN$, let $P$ be an order $\alpha$ digital $(t,m,s)$-net over $\FF_b$ with some integer $0\le t\le \alpha m$. Then, for any $\alpha' \in \NN$ with $1\le \alpha' <\alpha$, $P$ is also an order $\alpha'$ digital $(t_{\alpha'},m,s)$-net over $\FF_b$ with $t_{\alpha'}=\lceil t\alpha'/\alpha\rceil$.
\end{lemma}

Dick \cite{Dick07,Dick08} proposed the following digit interlacing composition to obtain explicit construction of higher order digital nets over $\FF_b$: For $m,s,\alpha\in \NN$, let $Q\subset [0,1)^{\alpha s}$ be a digital net over $\FF_b$ with generating matrices $C_1,\ldots,C_{\alpha s}\in \FF_b^{m\times m}$. For $1\le i\le m$ and $1\le j\le \alpha s$, we denote by $\bsc_{i,j}$ the $i$-th row vector of $C_j$. We now construct a digital net $P\subset [0,1)^s$ over $\FF_b$ with generating matrices $D_1,\ldots,D_s\in \FF_b^{\alpha m\times m}$ such that the ($\alpha(h-1)+i$)-th row vector of $D_j$ equals $\bsc_{h,\alpha(j-1)+i}$ for all $1\le h\le m$, $1\le i\le \alpha$ and $1\le j\le s$. Regarding this construction algorithm, we have the following, see for instance \cite[Corollary~3.4]{BDP11}.
\begin{lemma}\label{lem:ho_digital_net}
Let $Q$ be an order 1 digital $(t',m,\alpha s)$-net over $\FF_b$ with $0\le t'\le m$. Then a digital net $P$ constructed as above is an order $\alpha$ digital $(t,m,s)$-net over $\FF_b$ with
\begin{align}\label{eq:order_alpha_t-value}
 t = \alpha \min \left\{m, t'+ \left\lfloor \frac{s(\alpha-1)}{2} \right\rfloor\right\} .
\end{align}
\end{lemma}
\noindent
Thus in order to obtain an order $\alpha$ digital $(t,m,s)$-net with small $t$-value, we need an order 1 digital $(t',m,\alpha s)$-net with small $t'$-value. Here we recall that there have been many explicit constructions of order $1$ digital sequences (defined below) over $\FF_b$ for arbitrary dimension proposed in the literature, so that we can construct order 1 digital $(t',m,\alpha s)$-nets with small $t'$-value.
\begin{definition}\label{def:digital_seq}
Let $C_1,\ldots,C_s\in \FF_b^{\NN\times \NN}$ be $\NN\times \NN$ matrices over $\FF_b$. For $C_j=(c_{j,k,l})_{k,l\in \NN}$ we assume that there exists a function $K: \NN\to \NN$ such that $c_{j,k,l}=0$ when $k>K(l)$. Let $h$ be a non-negative integer with its $b$-adic expansion $h=\sum_{i=0}^{a-1}\eta_i b^i$ for some $a\in \NN$. For $1\le j\le s$, let us consider
\begin{align*}
 x_{h,j} = \frac{\xi_{1,h,j}}{b}+\frac{\xi_{2,h,j}}{b^2}+\cdots ,
\end{align*}
where $\xi_{1,h,j},\xi_{2,h,j},\ldots$ are given by
\begin{align*}
 (\xi_{1,h,j},\xi_{2,h,j},\ldots)^{\top} = C_j (\eta_0,\eta_1,\ldots,\eta_{a-1},0,0,\ldots)^{\top}.
\end{align*}
The sequence $S=(\bsx_0,\bsx_1,\ldots)$ with $\bsx_h=(x_{h,1},\ldots,x_{h,s})$ is called an digital sequence over $\FF_b$ (with generating matrices $C_1,\ldots,C_s$).

Moreover, let $t$ be a non-negative integer. For $m\in \NN$, let $C_{j,m\times m}$ be the the left upper $m\times m$ sub-matrix of $C_j$. If for all $m>t$ the matrices $C_{1,m\times m},\ldots,C_{s,m\times m}$ generate an order 1 digital $(t,m,s)$-net over $\FF_b$, we call $S$ an order 1 digital $(t,s)$-sequence over $\FF_b$.
\end{definition}
\begin{remark}\label{rem:ho_digital_net}
As mentioned above, there are many explicit constructions of order $1$ digital sequences over $\FF_b$, see for instance \cite{Faure82,Nied88,NXbook,Sobol67}. We refer to \cite[Chapter~6]{DPbook} for more information on this topic. For $\alpha,s\in \NN$, let $S$ be an order 1 digital $(t',\alpha s)$-sequence over $\FF_b$ with generating matrices $C_1,\ldots,C_{\alpha s}$ for some non-negative integer $t'$. Now let us define $m_0:=t'+ \lfloor s(\alpha-1)/2 \rfloor$. When $m\ge m_0$, by using the result of Lemma~\ref{lem:ho_digital_net}, we see that the digital net $P\subset [0,1)^s$ constructed by the digit interlacing composition based on $C_{1,m\times m},\ldots,C_{\alpha s,m\times m}$ becomes an order $\alpha$ digital $(\alpha m_0, m, s)$-net. Here the value $\alpha m_0$ does not depend on $m$.
\end{remark}

\section{Proof of Theorem~\ref{thm:main_result}}\label{sec:upper}
Throughout this section, let $P$ be an order $\beta$ digital net over $\FF_b$ for $\beta \in \NN$. Here we prove Theorem~\ref{thm:main_result}, i.e., an upper bound on the RMS worst-case error of QMC integration over $P\oplus \bssigma$ in $\Hcal_{\alpha,\bsgamma}$ with respect to a randomly chosen $\bssigma\in [0,1)^s$ when $\beta\ge 2\alpha$. 

\subsection{Interpolation of Dick metric functions}
In this subsection, we discuss an interpolation property of Dick metric functions, which shall become a crucial tool in the proof of an upper bound on the RMS worst-case error.
\begin{lemma}\label{lem:mertic_inter}
Let $\alpha,\beta\in \NN$ with $1< \alpha \le \beta$. For any $\bsk\in \NN_0^s$, it holds that
\begin{align*}
 \mu_{\alpha}(\bsk) \ge \frac{\alpha-1}{\beta-1} \mu_{\beta}(\bsk)+\frac{\beta-\alpha}{\beta-1} \mu_1(\bsk) .
\end{align*}
\end{lemma}
\noindent For convenience, we shall in what follows write
\begin{align*}
 A_{\alpha\beta}= \frac{\alpha-1}{\beta-1}\quad \text{and}\quad B_{\alpha\beta}= \frac{\beta-\alpha}{\beta-1}  .
\end{align*}

\begin{proof}
Since $\mu_{\alpha}(\bsk)=\sum_{j=1}^{s}\mu_{\alpha}(k_j)$ for any $\alpha\in \NN$ and $\bsk=(k_1,\ldots,k_s)\in \NN_0^s$, it suffices to prove that the inequality
\begin{align*}
 \mu_{\alpha}(k) \ge A_{\alpha\beta}\mu_{\beta}(k)+B_{\alpha\beta}\mu_1(k) ,
\end{align*}
holds for any $k\in \NN_0$. As the result is trivial for $k=0$, we only consider the case $k\ge 1$ in the following. Let us denote the $b$-adic expansion of $k$ by $k=\kappa_1 b^{c_1-1}+\kappa_2 b^{c_2-1}+\cdots+\kappa_v b^{c_v-1}$ for some $v\ge 1$ such that $\kappa_1,\ldots,\kappa_v\in \{1,\ldots,b-1\}$ and $c_1>\cdots>c_v>0$. When $\beta > v$, we write $c_{v+1}=c_{v+2}=\cdots = c_{\beta}=0$. Then, we have
\begin{align*}
 \mu_{\alpha}(k) = \sum_{i=1}^{\alpha}c_i \ge c_1 + \sum_{i=2}^{\alpha}c_{\alpha} = \mu_1(k)+ (\alpha -1)c_{\alpha} ,
\end{align*}
as well as
\begin{align*}
 \mu_{\beta}(k) = \sum_{i=1}^{\beta}c_i \le \sum_{i=1}^{\alpha}c_i + \sum_{i=\alpha+1}^{\beta}c_{\alpha}=\mu_{\alpha}(k)+(\beta -\alpha)c_{\alpha}.
\end{align*}
By using the above two inequalities, we obtain
\begin{align*}
\frac{\mu_{\beta}(k)-\mu_{\alpha}(k)}{\beta -\alpha} \le c_{\alpha}\le \frac{\mu_{\alpha}(k)-\mu_1(k)}{\alpha -1} ,
\end{align*}
from which we can easily see that the result follows.
\end{proof}

\begin{remark}\label{rem:mertic_inter}
In the proof of the upper bound on the RMS worst-case error, which shall be given in the next subsection, the condition $B_{\alpha\beta}>1/2$ is necessary. This condition can be satisfied if and only if $\beta \ge 2\alpha$. This is why we assume $\beta \ge 2\alpha$ in Theorem~\ref{thm:main_result} and Corollary~\ref{cor:existence}.
\end{remark}

\subsection{Upper bound on the RMS worst-case error}
Finally, we prove an upper bound on the RMS worst-case error of QMC integration over $P\oplus \bssigma$ in $\Hcal_{\alpha,\bsgamma}$ with respect to a randomly chosen $\bssigma\in [0,1)^s$. The following lemma stems from the proof of \cite[Theorem~30]{BD09}.
\begin{lemma}\label{lem:mse_Baldeaux_Dick}
Let $P$ be a digital net over $\FF_b$ and $P^{\perp}$ its dual net. For $u\subseteq \{1:s\}$, we write $P_u^{\perp}=\{\bsk_u\in \NN^{|u|}: (\bsk_u,\bszero)\in P^{\perp}\}$. The mean square worst-case error of QMC integration over $P\oplus \bssigma$ in $\Hcal_{\alpha,\bsgamma}$ with respect to a randomly chosen $\bssigma\in [0,1)^s$ is bounded by
\begin{align*}
 \left(e^{\rms\text{--}\wor}(\Hcal_{\alpha,\bsgamma};P)\right)^2 \le \sum_{\emptyset \ne u\subseteq \{1:s\}}\gamma_u D_{\alpha,b}^{|u|}\sum_{\bsk_u\in P_u^{\perp}}b^{-2\mu_{\alpha}(\bsk_u)} ,
\end{align*}
where we simply write $\mu_{\alpha}(\bsk_u):=\sum_{j\in u}\mu_{\alpha}(k_j)$ for $\emptyset \ne u\subseteq \{1:s\}$ and $\bsk_u\in \NN^{|u|}$, and $D_{\alpha,b}>0$ depends only on $\alpha$ and $b$ and is given by
\begin{align*}
 D_{\alpha,b} = \max_{1\leq v\leq \alpha}\left\{ \sum_{\tau=v}^{\alpha}\frac{(C_{\tau,b})^2}{b^{2(\tau-v)}}+\frac{2C_{2\alpha, b}}{b^{2(\alpha-v)}}\right\},
\end{align*}
with
\begin{align*}
 C_{1,b}=\frac{1}{2\sin(\tau/b)} \quad \text{and} \quad C_{\tau,b}=\frac{(1+1/b+1/(b(b+1)))^{\tau-2}}{(2\sin(\tau/b))^{\tau}}\quad \text{for $\tau\geq 2$}.
\end{align*}
\end{lemma}

In the subsequent analysis, we shall use the following inequality, see \cite[Lemma~13.24]{DPbook} for its proof.
\begin{lemma}\label{lem:binom_sum}
For any real number $b>1$ and any $k,t_0\in \NN$, we have
\begin{align*}
 \sum_{t=t_0}^{\infty}b^{-t}\binom {t+k-1}{k-1}\le b^{-t_0}\binom {t_0+k-1}{k-1}\left( 1-\frac{1}{b}\right)^{-k}.
\end{align*}
\end{lemma}

Now we are ready to prove Theorem~\ref{thm:main_result}.

\begin{proof}[Proof of Theorem~\ref{thm:main_result}]
Using Lemmas~\ref{lem:mertic_inter} and \ref{lem:mse_Baldeaux_Dick}, we have
\begin{align*}
 \left(e^{\rms\text{--}\wor}(\Hcal_{\alpha,\bsgamma};P)\right)^2 & \le \sum_{\emptyset \ne u\subseteq \{1:s\}}\gamma_u D_{\alpha,b}^{|u|} \sum_{\bsk_u\in P_u^{\perp}}b^{-2A_{\alpha\beta}\mu_{\beta}(\bsk_u)-2B_{\alpha\beta}\mu_1(\bsk_u)} \\
 & \le \sum_{\emptyset \ne u\subseteq \{1:s\}}\gamma_u D_{\alpha,b}^{|u|} \sum_{\bsk_u\in P_u^{\perp}}b^{-2A_{\alpha\beta}\delta_\beta(P)-2B_{\alpha\beta}\mu_1(\bsk_u)} \\
 & = \sum_{\emptyset \ne u\subseteq \{1:s\}}\gamma_u D_{\alpha,b}^{|u|}b^{-2A_{\alpha\beta}\delta_\beta(P)}W_u^{1,2B_{\alpha\beta}}(P) ,
\end{align*}
where $\delta_\beta(P)$ is defined as in Definition~\ref{def:minimum_Dick_metric} and we write
\begin{align*}
 W_u^{1,2B_{\alpha\beta}}(P) := \sum_{\bsk_u\in P_u^{\perp}}b^{-2B_{\alpha\beta}\mu_1(\bsk_u)} ,
\end{align*}
for $\emptyset \ne u\subseteq \{1:s\}$. Since $\beta\geq 2\alpha$, we have $B_{\alpha\beta}>1/2$ as stated in Remark~\ref{rem:mertic_inter}. In the following we focus on the term $W_u^{1,2B_{\alpha\beta}}(P)$. Since $\mu_1(\bsk_u)$ is an integer no less than both $|u|$ and $\delta_1(P)$ for any $\bsk_u\in \NN^{|u|}$, we have
\begin{align*}
 W_u^{1,2B_{\alpha\beta}}(P) & = \sum_{h=\max\{\delta_1(P),|u|\}}^{\infty}\sum_{\substack{\bsk_u\in P_u^{\perp}\\ \mu_1(\bsk_u)=h}}b^{-2B_{\alpha\beta}h} \\
 & = \sum_{h=\max\{\delta_1(P),|u|\}}^{\infty}b^{-2B_{\alpha\beta}h}\sum_{\substack{\bsk_u\in P_u^{\perp}\\ \mu_1(\bsk_u)=h}}1 \\
 & = \sum_{h=\max\{\delta_1(P),|u|\}}^{\infty}b^{-2B_{\alpha\beta}h}\sum_{\substack{\bsl_u\in \NN^{|u|}\\ |\bsl_u|_1=h}}\sum_{\substack{\bsk_u\in P_u^{\perp}\\ \mu_1(k_j)=l_j, \forall j\in u}}1,
\end{align*}
where we denote $|\bsl_u|_1=\sum_{j\in u}l_j$. For the innermost sum in the last expression, it is known from \cite[Lemma~13.8]{DPbook}\footnote{Although \cite[Lemma~13.8]{DPbook} only consider the case where $P$ is a digital net over $\FF_b$ with generating matrices of size $m\times m$, the proof still goes through even when $P$ is a digital net over $\FF_b$ with generating matrices of size $n\times m$ as long as $n\ge m$.} that we have
\begin{align*}
 \sum_{\substack{\bsk_u\in P_u^{\perp}\\ \mu_1(k_j)=l_j, \forall j\in u}}1 \le \begin{cases}
 0 & \text{if $|\bsl_u|_1< \delta_1(P)$,} \\
 (b-1)^{|u|} & \text{if $\delta_1(P)\le |\bsl_u|_1< \delta_1(P)+|u|$,} \\
 (b-1)^{|u|}b^{|\bsl_u|_1-(\delta_1(P)+|u|-1)} & \text{if $|\bsl_u|_1\ge \delta_1(P)+|u|$.}
 \end{cases}
\end{align*}
Thus $W_u^{1,2B_{\alpha\beta}}(P)$ can be bounded by
\begin{align}\label{eq:W_u_bound}
 W_u^{1,2B_{\alpha\beta}}(P) & \le \sum_{h=\max\{\delta_1(P),|u|\}}^{\delta_1(P)+|u|-1}b^{-2B_{\alpha\beta}h}\sum_{\substack{\bsl_u\in \NN^{|u|}\\ |\bsl_u|_1=h}}(b-1)^{|u|} \nonumber \\
 & \qquad + \sum_{h=\delta_1(P)+|u|}^{\infty}b^{-2B_{\alpha\beta}h}\sum_{\substack{\bsl_u\in \NN^{|u|}\\ |\bsl_u|_1=h}}(b-1)^{|u|}b^{|\bsl_u|_1-(\delta_1(P)+|u|-1)} \nonumber \\
 & = (b-1)^{|u|}\Biggl[ \sum_{h=\max\{\delta_1(P),|u|\}}^{\delta_1(P)+|u|-1}b^{-2B_{\alpha\beta}h}\binom {h-1}{|u|-1} \nonumber \\
 & \qquad + b^{-(\delta_1(P)+|u|-1)}\sum_{h=\delta_1(P)+|u|}^{\infty}b^{-(2B_{\alpha\beta}-1)h}\binom {h-1}{|u|-1}\Biggr].
\end{align}
For the second sum in (\ref{eq:W_u_bound}), we have
\begin{align*}
 & \quad \sum_{h=\delta_1(P)+|u|}^{\infty}b^{-(2B_{\alpha\beta}-1)h}\binom {h-1}{|u|-1} \\
 & = \sum_{h=\delta_1(P)}^{\infty}b^{-(2B_{\alpha\beta}-1)(h+|u|)}\binom {h+|u|-1}{|u|-1} \\
 & \le b^{-(2B_{\alpha\beta}-1)(\delta_1(P)+|u|)}\binom {\delta_1(P)+|u|-1}{|u|-1}\left( 1-b^{-(2B_{\alpha\beta}-1)}\right)^{-|u|} \\
 & \le \left(\frac{1}{b^{2B_{\alpha\beta}-1}-1}\right)^{|u|}\frac{(\delta_1(P)+1)^{|u|-1}}{b^{(2B_{\alpha\beta}-1)\delta_1(P)}},
\end{align*}
where we used Lemma~\ref{lem:binom_sum} in the first inequality as we have $B_{\alpha\beta}>1/2$ by the assumption $\beta\ge 2\alpha$, and the second inequality stems from the inequality
\begin{align*}
 \binom {\delta_1(P)+|u|-1}{|u|-1} =\prod_{i=1}^{|u|-1}\frac{\delta_1(P)+|u|-i}{|u|-i}\le (\delta_1(P)+1)^{|u|-1}.
\end{align*}
For the first sum in (\ref{eq:W_u_bound}), we have
\begin{align*}
 & \quad \sum_{h=\max\{\delta_1(P),|u|\}}^{\delta_1(P)+|u|-1}b^{-2B_{\alpha\beta}h}\binom {h-1}{|u|-1} \\
 & \le \sum_{h=\max\{\delta_1(P),|u|\}}^{\infty}b^{-2B_{\alpha\beta}h}\binom {h-1}{|u|-1} \\
 & = \sum_{h=\max\{\delta_1(P)-|u|,0 \}}^{\infty}b^{-2B_{\alpha\beta}(h+|u|)}\binom {h+|u|-1}{|u|-1} \\
 & \le b^{-2B_{\alpha\beta}\max\{\delta_1(P),|u|\}}\binom {\max\{\delta_1(P),|u|\}-1}{|u|-1}\left( 1-b^{-2B_{\alpha\beta}}\right)^{-|u|} ,
\end{align*}
where we used Lemma~\ref{lem:binom_sum} again in the last inequality. 

Now let us consider the case $\delta_1(P)\ge |u|$. The first sum in (\ref{eq:W_u_bound}) is bounded by
\begin{align*}
 \sum_{h=\max\{\delta_1(P),|u|\}}^{\delta_1(P)+|u|-1}b^{-2B_{\alpha\beta}h}\binom {h-1}{|u|-1} & \le b^{-2B_{\alpha\beta}\delta_1(P)}\binom {\delta_1(P)-1}{|u|-1}\left( 1-b^{-2B_{\alpha\beta}}\right)^{-|u|} \\
 & \le \left( \frac{b^{2B_{\alpha\beta}}}{b^{2B_{\alpha\beta}}-1}\right)^{|u|}\frac{(\delta_1(P)-1)^{|u|-1}}{b^{2B_{\alpha\beta}\delta_1(P)}} \\
 & \le \left( \frac{b^{2B_{\alpha\beta}}}{b^{2B_{\alpha\beta}}-1}\right)^{|u|}\frac{(\delta_1(P)+1)^{|u|-1}}{b^{2B_{\alpha\beta}\delta_1(P)}}.
\end{align*}
For the case $\delta_1(P)< |u|$, the first sum in (\ref{eq:W_u_bound}) is bounded by
\begin{align*}
 \sum_{h=\max\{\delta_1(P),|u|\}}^{\delta_1(P)+|u|-1}b^{-2B_{\alpha\beta}h}\binom {h-1}{|u|-1} & \le b^{-2B_{\alpha\beta}|u|}\left( 1-b^{-2B_{\alpha\beta}}\right)^{-|u|} \\
 & \le \left(\frac{b^{2B_{\alpha\beta}}}{b^{2B_{\alpha\beta}}-1}\right)^{|u|}\frac{(\delta_1(P)+1)^{|u|-1}}{b^{2B_{\alpha\beta}\delta_1(P)}}.
\end{align*}
Thus, regardless of whether $\delta_1(P)\ge |u|$ or $\delta_1(P)< |u|$, we have the bound on the first sum in (\ref{eq:W_u_bound}) as
\begin{align*}
 \sum_{h=\max\{\delta_1(P),|u|\}}^{\delta_1(P)+|u|-1}b^{-2B_{\alpha\beta}h}\binom {h-1}{|u|-1} \le \left(\frac{b^{2B_{\alpha\beta}}}{b^{2B_{\alpha\beta}}-1}\right)^{|u|}\frac{(\delta_1(P)+1)^{|u|-1}}{b^{2B_{\alpha\beta}\delta_1(P)}}.
\end{align*}
Combining this result with the bound on the second sum, we have
\begin{align*}
 W_u^{1,2B_{\alpha\beta}}(P)
 & \le (b-1)^{|u|}\Biggl[ \left(\frac{b^{2B_{\alpha\beta}}}{b^{2B_{\alpha\beta}}-1}\right)^{|u|}\frac{(\delta_1(P)+1)^{|u|-1}}{b^{2B_{\alpha\beta}\delta_1(P)}} \\
 & \qquad + b^{-(\delta_1(P)+|u|-1)}\left(\frac{1}{b^{2B_{\alpha\beta}-1}-1}\right)^{|u|}\frac{(\delta_1(P)+1)^{|u|-1}}{b^{(2B_{\alpha\beta}-1)\delta_1(P)}}\Biggr] \\
 & = G_{\alpha,\beta,b,u}\frac{(\delta_1(P)+1)^{|u|-1}}{b^{2B_{\alpha\beta}\delta_1(P)}} ,
\end{align*}
where we set 
\begin{align*}
 G_{\alpha,\beta,b,u} = (b-1)^{|u|}\left[ \left(\frac{b^{2B_{\alpha\beta}}}{b^{2B_{\alpha\beta}}-1}\right)^{|u|}+b\left( \frac{1}{b^{2B_{\alpha\beta}}-b}\right)^{|u|}\right].
\end{align*}
So far we have obtained
\begin{align*}
 & \quad \left(e^{\rms\text{--}\wor}(\Hcal_{\alpha,\bsgamma};P)\right)^2 \\
 & \le \sum_{\emptyset \ne u\subseteq \{1:s\}}\gamma_u D_{\alpha,b}^{|u|}b^{-2A_{\alpha\beta}\delta_\beta(P)}W_u^{1,2B_{\alpha\beta}}(P) \\
 & \le \frac{1}{b^{2A_{\alpha\beta}\delta_\beta(P)+2B_{\alpha\beta}\delta_1(P)}}\sum_{\emptyset \ne u\subseteq \{1:s\}}\gamma_u D_{\alpha,b}^{|u|}G_{\alpha,\beta,b,u}(\delta_1(P)+1)^{|u|-1}.
\end{align*}

Finally let us recall that $P$ is an order $\beta$ digital $(t,m,s)$-net over $\FF_b$. From this fact and Lemma~\ref{lem:propagation}, we have
\begin{align*}
 \delta_1(P) > m-t_1\quad \text{and}\quad \delta_\beta(P) > \beta m-t ,
\end{align*}
where $t_1=\lceil t/\beta\rceil$. Thus, we have
\begin{align*}
 2A_{\alpha\beta}\delta_\beta(P)+2B_{\alpha\beta}\delta_1(P) & > 2A_{\alpha\beta}(\beta m-t) + 2B_{\alpha\beta}(m-t_1) \\
 & = (2\beta A_{\alpha\beta}+2B_{\alpha\beta})m - 2A_{\alpha\beta}t - 2B_{\alpha\beta}t_1 .
\end{align*}
In the above, it holds that
\begin{align*}
 2\beta A_{\alpha\beta}+2B_{\alpha\beta} =2\alpha.
\end{align*}
Since $t_1=0$ is best possible, we have $\delta_1(P)\le m+1$. Therefore, we get
\begin{align*}
 \left(e^{\rms\text{--}\wor}(\Hcal_{\alpha,\bsgamma};P)\right)^2 & \leq \frac{b^{2A_{\alpha\beta}t + 2B_{\alpha\beta}t_1}}{b^{2\alpha m}}\sum_{\emptyset \ne u\subseteq \{1:s\}}\gamma_u D_{\alpha,b}^{|u|}G_{\alpha,\beta,b,u}(m+2)^{|u|-1} \\
 & \leq \frac{b^{2A_{\alpha\beta}t + 2B_{\alpha\beta}t_1}}{b^{2\alpha m}}\sum_{\emptyset \ne u\subseteq \{1:s\}}\gamma_u D_{\alpha,b}^{|u|}G_{\alpha,\beta,b,u}(3m)^{|u|-1} ,
\end{align*}
which completes the proof by using the inequality $(\sum_i a_i)^{1/2}\leq \sum_i a_i^{1/2}$ for $a_i\geq 0$ and by choosing 
\begin{align}\label{eq:constant_full}
C_{\alpha,\beta,b,u,t} = b^{A_{\alpha\beta}t + B_{\alpha\beta}t_1}D_{\alpha,b}^{|u|/2}G_{\alpha,\beta,b,u}^{1/2}\left( \frac{3}{\log b}\right)^{(|u|-1)/2}
\end{align}
for all $\emptyset \ne u\subseteq \{1:s\}$ such that the bound (\ref{eq:main_result}) holds.
\end{proof}



\begin{thebibliography}{99}
\bibitem{BD09} J. Baldeaux and J. Dick, QMC rules of arbitrary high order: Reproducing kernel Hilbert space approach, Constr. Approx., 30 (2009) 495--527.
\bibitem{BDP11} J. Baldeaux, J. Dick and F. Pillichshammer, Duality theory and propagation rules for higher order nets, Discrete Math., 311 (2011) 362--386.
\bibitem{Dick07} J. Dick, Explicit constructions of quasi-Monte Carlo rules for the numerical integration of high-dimensional periodic functions, SIAM J. Numer. Anal., 45 (2007) 2141--2176.
\bibitem{Dick08} J. Dick, Walsh spaces containing smooth functions and quasi-Monte Carlo rules of arbitrary high order, SIAM J. Numer. Anal., 46 (2008) 1519--1553.
\bibitem{Dick09} J. Dick, The decay of the Walsh coefficients of smooth functions, Bull. Aust. Math. Soc., 80 (2009) 430--453.
\bibitem{DNP14} J. Dick, D. Nuyens and F. Pillichshammer, Lattice rules for nonperiodic smooth integrands, Numer. Math., 126 (2014) 259--291.
\bibitem{DPbook} J. Dick and F. Pillichshammer, \emph{Digital Nets and Sequences: Discrepancy Theory and Quasi-Monte Carlo Integration}, Cambridge University Press, Cambridge, 2010.
\bibitem{Faure82} H. Faure, Discr\'{e}pances de suites associ\'{e}es \`{a} un syst\`{e}me de num\'{e}ration (en dimension s). Acta Arith., 41 (1982) 337--351.
\bibitem{HMOT15} A. Hinrichs, L. Markhasin, J. Oettershagen and T. Ullrich, Optimal quasi-Monte Carlo rules on higher order digital nets for the numerical integration of multivariate periodic functions, ArXiv preprint arXiv:1501.01800.
\bibitem{KNbook} L. Kuipers and H. Niederreiter, \emph{Uniform Distribution of Sequences}, John Wiley, New York, 1974.
\bibitem{Mark13} L. Markhasin, Quasi-Monte Carlo methods for integration of functions with dominating mixed smoothness in arbitrary dimension, J. Complexity, 29 (2013) 370--388.
\bibitem{Nied86} H. Niederreiter, Low-discrepancy point sets, Monatsh. Math., 102 (1986) 155--167.
\bibitem{Nied88} H. Niederreiter, Low-discrepancy and low-dispersion sequences. J. Number Theory, 30 (1988) 51--70.
\bibitem{Nbook} H. Niederreiter, \emph{Random Number Generation and Quasi-Monte Carlo Methods}, CBMS-NSF Regional Conference Series in Applied Mathematics, Vol.~63, SIAM, Philadelphia, 1992.
\bibitem{NXbook} H. Niederreiter and C.~P. Xing, \emph{Rational Points on Curves over Finite Fields: Theory and Applications}, London Mathematical Society Lecture Note Series, 285. Cambridge University Press, Cambridge, 2001.
\bibitem{RT97} M.~Yu. Rosenbloom and M.~A. Tsfasman, Codes for the $m$-metric, Probl. Inf. Transm., 33 (1997) 55--63.
\bibitem{SW98} I.~H. Sloan and H. Wo\'zniakowski, When are quasi-Monte Carlo algorithms efficient for high-dimensional integrals?, J. Complexity, 14 (1998) 1--33.
\bibitem{Sobol67} I.~M. Sobol', The distribution of points in a cube and approximate evaluation of integrals. Zh. Vycisl. Mat. i Mat. Fiz., 7 (1967) 784--802.
\bibitem{Teml03} V.~N. Temlyakov, Cubature formulas, discrepancy, and nonlinear approximation, J. Complexity, 19 (2003) 352--391.
\bibitem{Ullrich14} M. Ullrich, On ``Upper error bounds for quadrature formulas on function classes'' by K.~K. Frolov, ArXiv preprint arXiv:1404.5457.
\bibitem{Wbook} G. Wahba, \emph{Spline Models for Observational Data}, CBMS-NSF Regional Conference Series in Applied Mathematics, Vol.~59, SIAM, Philadelphia, 1990.
\end{thebibliography}
\end{document}